\documentclass{conm-p-l}

\usepackage{geometry}
\usepackage{latexsym}
\usepackage{amssymb,amsmath,amstext,amsthm,amsfonts}
\usepackage[latin1]{inputenc}
\usepackage{graphicx}
\usepackage{amsthm}

\newtheorem{theorem}{Theorem}[section]
\newtheorem{lemma}[theorem]{Lemma}
\newtheorem{proposition}[theorem]{Proposition}

\theoremstyle{definition}
\newtheorem{definition}[theorem]{Definition}

\theoremstyle{remark}
\newtheorem{remark}[theorem]{Remark}

\numberwithin{equation}{section}

\begin{document}

 \title[On Generalized Weil Representations over Involutive Rings]{On Generalized Weil Representations over Involutive Rings}


\author{Luis  Guti\'errez}
\address{Instituto de Matem\'aticas, Universidad Austral de Chile, Campus Isla Teja s/n, Valdivia, Chile.}
\curraddr{}
\email{luis.gutierrezfrez@docentes.uach.cl} 
\thanks{}

\author{Jos\'e Pantoja}
\address{Instituto de Matem\'aticas, Pontificia Universidad Cat\'olica de Valpara\'iso, Blanco Viel 596, Cerro Bar\'on, Valpara\' iso, Chile.}
\curraddr{}
\email{jpantoja@ucv.cl}
\thanks{  } 

\author{Jorge Soto-Andrade}
\address{Departamento of Matem\'aticas, Facultad de Ciencias, Universidad de  Chile, Casilla 653, Santiago, Chile.}
\curraddr{}
\email{sotoandrade@u.uchile.cl}
 
\thanks{The  authors were supported by Fondecyt Grant 1095078.}


\date{}
 

\date{}

\begin{abstract}
We  construct via generators and relations, generalized Weil representations  for   analogues of   classical $SL(2,k),  \;\; k$ a field, over involutive base rings $(A, \ast).$     This family of   groups covers different kinds of groups, classical and non classical.  We give some examples that include symplectic groups as well as        non classical groups like    $SL_\ast(2,A_m), $  where $A_m$ is the finite modular analogue of the algebra of real m-jets in one dimension with its canonical involutive symmetry. 

\end{abstract}

\maketitle

 
\section{Introduction}

The classification and construction of all unitary representations of a given group is a fundamental problem in   representation theory.  Especially interesting is the case of classical groups over a local field (particularly Lie groups) or over a finite field. It turns out sometimes that, remarkably enough, we can transpose methods of construction from Lie groups to finite groups of Lie type or vice versa. 
 
So we may try to use  finite groups of Lie type as a testing ground for designing    methods of construction that could be transposed to the case of Lie groups. In principle, we may expect to be more feasible to find                                                                                                                                                                                                                                                                                                                                                                                                                                             for the former,  methods that are elementary, uniform  and universal in the sense that they can be carried over to  the latter.
  
  This is however  a hard task and one of the few finite groups of Lie type   for which this has been achieved up to now,  is   $G = SL(2, \mathbb F_q) $, all of whose irreducible representations may be constructed elementary and uniformly by by constructing  a remarkable representation of $G$  
associated to each semi-simple two-dimensional (Galois) algebra over the finite field 
$\mathbb F_q$
(namely  $\mathbb F_q \times\mathbb F_q
$ and   $ \mathbb F_{q^2}, $ up to isomorphism), which gives by decomposition according to the
orthogonal group of the  corresponding Galois norm, the principal and the cuspidal series of irreducible 
representations of  $G,$  respectively \cite{sa1}. 

We recall that these representations are called nowadays Weil (or Shale-Weil) representations of $G$, because they arose from Weil's classical construction \cite {weil}, for any local field,  of the ``oscillator representation''  arising in the work of Shale \cite{shale} on linear symmetries of free boson fields in quantum field theory. Interestingly enough, this construction arising in quantum physics enabled us in the sixties to solve the fundamental problem for $SL(2, \mathbb F_q) $,   had remained open since the construction of its character table at the turn of the century. 

Originally  Shale  constructed   the oscillator  representation of  $Sp(2n,k)$ , for  $k = \mathbb R$, which turns out to be a projective representation,    taking advantage of the representation theory of the Heisenberg group $H_n$ in  $n$ degrees of freedom, as described by the Stone von Neumann Theorem, that says that  $ H_n$ has just one irreducible unitary representation of dimension greater than $1$ with a given central character (the famous  Schr\"odinger representation).  Weil \cite{weil} extended later this construction to the local field case (for a very readable and more geometric account, see \cite{lv}). 
Shortly thereafter,  Cartier noticed that Weil representations associated to any non-degenerate quadratic form $Q$ over the base field, could be constructed for  $ SL(2,k)$ in an elementary way, via generators and relations, using the well known presentation of this group, which follows from its Bruhat decomposition. In this way the Weil representation appears as a functorial construction on the category of quadratic spaces.   The projective oscillator representation corresponds then to the rank one quadratic form  $x^2$ over $k$ and we get in fact (true) Weil representations  in the even rank case \cite{sa1}.  Moreover, under this approach the intertwining of the Weil representation $W_Q$ associated to the quadratic form  $Q$ is ``explained" by the natural action of the orthogonal group  $O(Q) $ in the space of  $W_Q$.
Cartier conjectured further  that this construction would provide a uniform and elementary method for constructing all irreducible representations of $SL(2,k,)$ and for  $ Sp(2n,k)$ as well, although the presentation known at the time for  $ Sp(2n,k)$ (found by Dickson at the turn of the century) was rather unyielding.  
A more convenient presentation based on the Bruhat decomposition,  was found however in \cite{sa1} by regarding $Sp(2n,k) $ as an   $``SL(2)"$ group over the involutive  ring $A = M(n,k)$ (with transpose as involution).  An elementary construction of these Weil representations for the symplectic groups became then possible, via generators and relations. 

This suggests to try to construct  Weil representations by generators and relations in a very general setting, for an analogue $G$ of  $ SL(2,k) $ over any involutive ring, for which a ``Bruhat" presentation analogue to the classical one holds.  The method of construction  would involve  defining suitable ``Weil operators" associated to the generators and checking  that the defining relations of the presentation are preserved.   

The central result of this paper is the construction of a very general Weil representation  from  abstract core data consisting of  a module  $M$ over an involutive ring  $(A,\ast)$ equipped with a suitable  non degenerate complex valued self pairing  $\chi$ and its  second order homogeneous companion $\gamma$.   
We recover as particular examples   of our  construction, the Weil representations  of   the  symplectic groups $Sp(2n,k)$ of \cite{sa1}  and   the generalized Weil representation  in \cite{lucho}, for  the non classical case  of an involutive base ring having a nilpotent radical.  

We also give below a a first general decomposition of this generalized Weil representation, based on the symmetry group of our data.   
 
We may remark that there is an intriguing similarity of our  involutive analogues and quantum analogues of classical  $S L(2,k)$, seen as matrix groups with non commuting entries, that could be ``explained" as follows.  
We are working in fact in a tamely non commutative case,   in which non commutativity  is ``controlled'' by an involution  $T: a\mapsto a^{\ast }$ 
in the coefficient ring  $A.$ and we require our entries to  $\ast-$commute.   Indeed letting    
  \ $m:A\otimes A\rightarrow A$   the multiplication
of \ $A$ \ and \ \ $S:A\otimes A\rightarrow A\otimes A$ \ \   the ``flip''
\ $x\otimes y\longmapsto y\otimes x,$ we see that \ $T$ \ ``transforms'' \ $m
$ $\ $into $\ m\circ S$ as follows: \ \ \ \ \ $m\circ S=T\circ m\circ (T\otimes T)^{-1},$
just as the $R-$matrix ''controls'' the
lack of co-commutativity in a quantum group, i. e. \ \ \ $S\circ \Delta = 
\mathcal{I}_{R}\circ \Delta ,$\ \ where $\Delta $ stands for the
comultiplication in the corresponding Hopf algebra and   \ $\mathcal{I}_{R}
$ \ denotes conjugation by the \ $R-$matrix, which is an invertible element
in    \ $A\otimes A.$\ 
Our 
 \ $\ast -$ determinant  $
ad^{\ast }-bc^{\ast },$ defined  only  for matrices
whose entries \ ``$\ast -$commute'' . follows,  in a way closely reminiscent of  the $q-$determinant.   
  
Regarding the existence of Bruhat presentations for involutive analogues of  $\ SL_{\ast }(2,A)$, we mention that a first step  was already accomplished in  \cite{PSAjalg}, where a {\em Bruhat decomposition} for these groups was obtained in the case of an artinian
involutive base ring $A$.  Later, the classical Bruhat presentation of $SL(2,k)$
was extended to the case of an artinian simple involutive $A$ in \cite{P}. Then
the existence of a Bruhat presentation for  $\ SL_{\ast }(2,A)$ was proved in \cite{PSAcom} for a wide class of involutive rings  $A$, namely, 
those admitting a weak  $\ast -$analogue of the euclidean algorithm for coprime elements,   that includes the artinian simple rings considered in  \cite{P} as well as rings as $\mathbb Z$.  
We notice that in this last paper, $\varepsilon-$analogues,    ($ \epsilon = \pm 1$)    besides  $\ast-$analogues, were considered for  $SL(2,k)$, with  $\varepsilon = 1$ corresponding to the previous case, so that now we recover split orthogonal groups as well as symplectic groups.

This article is organized as follows.

In section 2 we recall the definition and main properties of the $\ast$ and $\varepsilon$ analogues of classical  $SL(2,k)$.

In section 3, we recall Bruhat decompositions for our groups   $SL_{\ast}^{\varepsilon}(2,A)$.

In section 4, we construct generalized Weil representations for   $SL_{\ast}^{\varepsilon}(2,A)$.

In section 5, we show, for the case of symplectic groups, how to recover the Weil representations constructed in \cite{sa1} and we construct also Weil representations for the case of an odd rank quadratic form, not considered there.

In section 6, we show how to recover the Weil representation constructed in the non classical case of an involutive ring with non trivial nilpotent Jacobson radical in \cite{lucho}.

 Finally, in section 7, we give a first decomposition of our generalized Weil representations, in terms of the irreducible representations of the defining data  $\chi$ and  $\gamma$.

\section{The groups  $SL_{\ast}^\varepsilon(2,A)$}
\label{recall}
We recall in this section the definition and main properties of the generalized classical 
groups $SL_{\ast}^{\varepsilon}(2,A)$, where   $(A,\ast)$ is a unitary   ring  with involution $\ast$ and   $\varepsilon=\pm1 \in A$.  We look upon these groups as  $\ast$  and $\varepsilon$-analogues of  the groups  $SL(2,k)$, $k$ a field.  For more details
see \cite{PSAjalg}.

\begin{definition}

For a ring $A$ with involution  $\ast$, the group $SL_{\ast}^{\varepsilon}(2,A)$ $\ $ may be defined as the set  of all   $2 \times 2$ matrices 
$  
\left(
\begin{array}
[c]{cc}
a & b\\
c & d
\end{array} 
\right)  $ with    $ a,b,c, d \in A$ such that 

\begin{enumerate}    
\item $ab^{\ast}=-\varepsilon ba^{\ast},$
\item $cd^{\ast}= -\varepsilon dc^{\ast},$
\item $a^{\ast}c=-\varepsilon c^{\ast}a,$
\item $b^{\ast}d=-\varepsilon d^{\ast}b,$
\item $ad^{\ast}+\varepsilon bc^{\ast}=a^{\ast}d+\varepsilon c^{\ast}b=1 $
\end{enumerate} 
  with   matrix multiplication.

\end{definition}

We notice that this group may also be described as the unitary group $ U(H_\varepsilon) $ of the  $\varepsilon-$Hermitian form 
$  H_ \varepsilon $ on  $ A^2 = A \times A$ associated to  the $ 2 \times 2 $ matrix
$J_{\varepsilon} =\left(
\begin{array}
[c]{cc}%
0 & 1\\
\varepsilon & 0
\end{array}
\right),$ in the following sense.

\begin{definition}

 A left $\varepsilon-$Hermitian form $H$  on  a left  $A$-module  $M$ is a function     \mbox {$H: M\times M \rightarrow
A$}  that is biadditive, left linear in the first variable and such that
$H(y,x)=\varepsilon H(x,y)^{\ast},$ for all  $x,y \in M.$

\end{definition}

Recall that when the $A-$module $M$ is free of finite rank $m$, we have a matrix description of left $\varepsilon-$Hermitian forms, as follows.  
 We extend first the involution $\ast$ in $A$ to the  full matrix  ring $M(m,A)$ of all  $m \times m$ matrices with coefficients in  $A$, putting $T^{\ast}=( t_{ji}^{\ast})_{ 1 \leq i, j  \leq m}$ for any   
    $T=(t_{ij}))_{ 1 \leq i, j  \leq m}   \in M(m,A).  $                                                                                                                                                                                                                                                                                                                                                                                                                                                                                                                                                                                                                                                                                                                                  
Now, a basis   $ {\mathcal B} =  \{ e_{1},...,e_{m} \}  $   of  the free $A-$module  $M$  having been chosen, 
we define the matrix $[H]$ of   $H$ with respect to $\mathcal B$ 
   by $[H]=(H(e_{i},e_{j}))_{ 1 \leq i, j  \leq m}$.

Then $H(u,v)=u[H]v^\ast \; \; \; \;    ( u, v \in M) $ and 
 conversely, given a matrix $T$ such that $T^{\ast
}=\varepsilon T$ we recover an $\varepsilon-$Hermitian form $H_T$  by $H_{T}
(u,v)=uTv^\ast \;\;\;  (u, v  \in  M).$  Notice that above we have written just  $H_\varepsilon$ instead of 
$H_{J_\varepsilon}$ for the associated $\varepsilon-$hermitian form to the matrix  $J_\varepsilon.$


\begin{proposition}
The group $SL_{\ast}^{\varepsilon}(2,A)$  may be  defined equivalently as the set of all automorphisms $g$ of the  $A-$module $M = A \times A$ such that  $ H_\varepsilon \circ(g\times g)=H_\varepsilon$ or 
  in matrix form as
$$    
SL_{\ast}^{\varepsilon}(2,A)=\{T\in M(2,A) \ | \ TJ_\varepsilon T^{\ast}=J_\varepsilon\}. $$
 \end{proposition}

\begin{flushright}
$\square$
\end{flushright}
\bigskip

\section { Bruhat presentations of $SL_{\ast}^{\varepsilon}(2,A)$}
Given a unitary ring $A$ with an involution $*$, we will write $A^{sym}$ to denote the set of   all $\varepsilon-$symmetric  elements in $A$, i. e., elements $a\in A$ such that $a^*=-\varepsilon a$. We set\\
 
 $h_{t}=\left(
\begin{array}
[c]{cc}%
t & 0\\
0 & t^{\ast-1}%
\end{array}
\right)  $ ($t\in A^{\times}$), $w=w_{\varepsilon}=\left(
\begin{array}
[c]{cc}%
0 & 1\\
\varepsilon & 0
\end{array}
\right)  $ and $u_{s}=\left(
\begin{array}
[c]{cc}%
1 & s\\
0 & 1
\end{array}
\right)  $ ($s\in A^{sym}$)

\bigskip

\begin{definition} 

We will say that $G=SL_{\ast}^{\varepsilon}(2,A)$ has a Bruhat presentation (see \cite{PSAcom, lucho}) if
it is generated by the above elements with  defining relations
\begin{enumerate}
\item $\,h_{t}h_{t^{\prime}}=h_{tt^{\prime}}$,   $\,u_{b}u_{b^{\prime}}=u_{b+b^{\prime
}};$
\item $w^{2}=h_{\varepsilon}$;

\item  $h_{t}u_{b}=u_{tbt^{\ast}}h_{t}$;
\item $w h_{t}=h_{{t^\ast}^{-1}}w$;
\item $w u_{t^{-1}}w u_{-\varepsilon t}w u_{t^{-1}}
=h_{-\varepsilon t}$, with $t$ an invertible  $\varepsilon- $symmetric  element in $A$.
\end{enumerate}
\end{definition}

 \begin{remark}
 The relations above are called ``universal" in loc. cit. because they hold for every possible involutive base ring $(A, \ast).$
\end{remark}

\section{ A generalized Weil representation for  $G=SL_{\ast}^{\varepsilon}(2,A).$ }
In this section $A$ denotes  a ring $A$ with an involution $\ast$, i.e. an
antiautomorphism of  $A$ of order 2. In what follows we will assume that the ring $A$ is finite and also that the  group  $G=SL_{\ast}^{\varepsilon}(2,A)$ has a Bruhat presentation.

\subsection{Data for constructing a Weil representation of   $G=SL_{\ast}^{\varepsilon}(2,A)$}
\label{data}
We will construct a generalized Weil representation for $G$, associated to the following data:
\begin{enumerate}
\item  A finite right $A$-module $M$.

\item A bi-additive function $\chi:M\times M\rightarrow\mathbb{C}^{\times}$ and a character  $\alpha\in\widehat{A^{\times}}$ 
such that:

\begin{enumerate}
\item $ \chi(xt, y) = \alpha(tt^{\ast})\chi(x, yt^\ast) $ for $x,y\in M$ and $t\in
A^{\times}$   ($\chi$ is  $\alpha$-balanced).
\item $\chi(y,x)=[\chi(x,y)]^{-\varepsilon}$.
We observe that $[\chi(x,y)]^{-\varepsilon}=\chi(-\varepsilon x,y)$ \;\;\; ($\chi$ is  $\varepsilon$-symmetric).
\item $\chi(x,y)=1$ for any $x\in M$, implies $y=0$\;\;\; ($\chi$ is non-degenerate).
\end{enumerate}

\item  A function  $\gamma:A^{sym}\times M\rightarrow\mathbb{C}^{\times}$
such that:
\begin{enumerate}
\item $\gamma(b+b^{\prime},x)=\gamma(b,x)\gamma(b^{\prime},x)$, for all $b, b'\in A^{sym}$ and $x\in M$.
\item $\gamma(b,xt)=\gamma(tbt^{\ast},x)$ or equivalently $\gamma(b,x)=\gamma
(tbt^{\ast},xt^{-1})$, for all  $b\in A^{sym}$, $t\in A^{\times}$ and  $x\in M$.
\item  $\gamma(t,x+z)=\gamma(t,x)\gamma(t,z)\chi(x,zt)$ for all   $x,z \in M, t\in A^{sym}.$
 where $t$ is  $\varepsilon-$ symmetric invertible in $A$ and  $c\in\mathbb{C}^{\times}$ satisfies $c^{2}\left\vert M\right\vert =\alpha(\varepsilon)$.
\end{enumerate}

\item 
  We assume moreover that these data are related by the equation:
\begin{equation}
c\gamma(-\varepsilon t,x)\underset{y\in M}{\sum}\chi(-\varepsilon
x,y)\gamma(t^{-1},y)=\alpha(-t).
\end{equation}
\end{enumerate}

\begin{lemma}
 If $t$ is an $\varepsilon-$symmetric invertible element in $A$, then $\gamma(-\varepsilon t,x)=\gamma(t^{-1},xt)$.
\end{lemma}
\begin{proof} Since $t$ is $\varepsilon$-symmetric we have $t^*=-\varepsilon t$. Hence using the second condition  on $\gamma$ we get $\gamma(t^{-1},xt)=\gamma(tt^{-1}t^*,x)=\gamma(-\varepsilon t,x)$ and our lemma follows.
\end{proof}
\begin{proposition}

We consider as above $\chi$, $\gamma$, $\alpha$  and $c$ satisfying $c^2\vert M\vert=\alpha(\varepsilon)$. 
The following two relations are equivalent:
\begin{enumerate}
\item \ $c\gamma(-\varepsilon t,x)\underset{y\in M}{\sum}\chi(-\varepsilon
x,y)\gamma(t^{-1},y)=\alpha(-t)$.
\item $\underset{y\in M}{\sum}\gamma(t,y)=\frac{\alpha(\varepsilon t)}{c}$.
\end{enumerate}
\end{proposition}  
\begin{proof}
The relation 
$c\gamma(-\varepsilon t,x)\underset{y\in M}{\sum}\chi(-\varepsilon
x,y)\gamma(t^{-1},y)=\alpha(-t)$ is equivalent to

$\underset{y\in M}{\sum}\gamma(t^{-1},y)\gamma(-\varepsilon t,x)\chi
(-\varepsilon x,y)=\frac{\alpha(-t)}{c}.$

Now, given that $t^{\ast}=-\varepsilon t$ and that $\gamma(-\varepsilon
t,x)=\gamma(t^{-1},xt),$ we get that the last above relation is equivalent to

$\underset{y\in M}{\sum}\gamma(t^{-1},y)\gamma(t^{-1},xt)\chi(-\varepsilon
x,y)=\frac{\alpha(-t)}{c}$. Using the properties of $\gamma$ this last
expresion is the same as

$\underset{y\in M}{\sum}\frac{\gamma(t^{-1},y+xt)}{\chi(y,x)}\chi(-\varepsilon
x,y)=\frac{\alpha(-t)}{c}.$

But given that $\frac{\chi(-\varepsilon x,y)}{\chi(y,x)}=1,$ we are reduced to

$\underset{y\in M}{\sum}\gamma(t^{-1},y+xt)=\frac{\alpha(-t)}{c},$
\noindent
from which the result follows.

\end{proof}

\subsection{Construction of the generalized Weil representation of  $G$.}
In the theorem below we keep the notations and hypotheses introduced in section \ref{data}.

\begin{theorem}
\label{gwrbasis}

There is a representation $( \mathbb{C}^{M},\rho)$ of $G$ such that 

$\rho_{u_{b}}(e_{x})=\gamma(b,x)e_{x}$

$\rho_{h_{t}}(e_{x})=\alpha(t)e_{xt^{-1}}$

$\rho_{w}(e_{x})=c\underset{y\in M}{\sum}\chi(-\varepsilon x,y)e_{y}$

\noindent
for  $x \in M, b \in A^\times \cap A^\times , t \in A^\times,$   where      $   e_x  $ denotes Dirac delta function at  $x \in M$  given by   $e_x(y) = 1$ if   $y = x;  \;\; e_x(y) = 0$ otherwise. 

This representation is called the generalized Weil representation of  $SL^{\varepsilon}_*(2,A)$ associated to the data 
$(M, \alpha, \gamma, \chi).$
\end{theorem}

\begin{proof}

It is enough to verify that $\rho_{u_{b}},\rho_{h_{t}},\rho_{w}$ defined as
above, satisfy the relations corresponding to the universal relations in the Bruhat 
presentation of $G.$

To this end, we observe that
\begin{enumerate}
\item[(i)]
($\rho_{h_{t}}\circ\rho_{h_{t^{\prime}}})(e_{x})=\rho_{h_{t}}(\alpha
(t^{\prime})e_{xt^{\prime-1}})=\alpha(t^{\prime})\alpha(t)e_{xt^{\prime
-1}t^{-1}}=\alpha(t^{\prime})\alpha(t)e_{x(tt^{\prime})^{-1}}$

$=\alpha(tt^{\prime})e_{x(tt^{\prime})^{-1}}=\rho_{h_{tt^{\prime}}}(e_{x})$

\item[(ii)]

$(\rho_{u_{b}}\circ\rho_{u_{b^{\prime}}})(e_{x})=\gamma(b^{\prime}%
,x)\gamma(b,x)e_{x}=\gamma(b+b^{\prime},x)e_{x}=(\rho_{u_{b+b^{\prime}}%
})(e_{x})$

\item[(iii)]

$(\rho_{w}\circ\rho_{w})(e_{x})=c^{2}\underset{z\in M}{\sum}\underset{y\in
M}{\sum}\chi(-\varepsilon x,y)\chi(-\varepsilon y,z)e_{z}=\alpha
(\varepsilon)e_{x\varepsilon}=\rho_{h_{\varepsilon}}e_{x}$

\item[(iv)]

$(\rho_{h_{t}}\circ\rho_{u_{b}})(e_{x})=\gamma(b,x)\alpha(t)e_{xt^{-1}}%
=\alpha(t)\gamma(tbt^{\ast},xt^{-1})e_{xt^{-1}}=(\rho_{tbt^{\ast}}\circ
\rho_{h_{t}})(e_{x})$

\item[(v)]

$\rho_{w}\circ\rho_{h_{t}})(e_{x})=c\alpha(t)\underset{y}{\sum}\chi
(-\varepsilon xt^{-1},y)e_{y}=c\underset{y}{\sum}\chi(-\varepsilon
x,yt^{\ast^{-1}})\alpha(t^{\ast^{-1}})e_{y}=$\newline$\rho_{h_{t^{\ast^{-1}}}%
}(c\underset{y}{\sum}\chi(-\varepsilon x,y)e_{y})=(\rho_{h_{t^{\ast^{-1}}}%
}\circ\rho_{w})(e_{x})$

\item[(vi)]  Finally, a computation shows that

($\rho_{w}\rho_{u_{t^{-1}}}\rho_{w}\rho_{u_{-\varepsilon t}})(e_{x}%
)=\underset{z}{\sum}c^{2}\gamma(-\varepsilon t,x)\left(  \underset{y}{\sum
}\chi(-\varepsilon x,y)\gamma(t^{-1},y)\chi(-\varepsilon y,z)\right)  e_{z}$

$(\rho_{h_{-\varepsilon t}}\rho_{u_{-t^{-1}}}\rho_{w}\rho_{h_{\varepsilon}%
})(e_{x})=c\alpha(-t)\underset{z}{\sum}\chi(-\varepsilon x,-\varepsilon
zt)\gamma(-t^{-1},-\varepsilon zt)e_{z}.$

So we want to prove that

$c\gamma(-\varepsilon t,x)\underset{y}{\sum}\chi(-\varepsilon x,y)\gamma
(t^{-1},y)\chi(-\varepsilon y,z)=\alpha(-t)\chi(-\varepsilon x,-\varepsilon
zt)\gamma(-t^{-1},-\varepsilon zt)$

But by hypothesis

$c\gamma(-\varepsilon t,x-\varepsilon z)\underset{y\in M}{\sum}\chi
(-\varepsilon(x-\varepsilon z),y)\gamma(-\varepsilon t,yt^{-1})=\alpha(-t),$

so

$c\gamma(-\varepsilon t,x)\underset{y}{\sum}\chi(-\varepsilon x,y)\gamma
(t^{-1},y)\chi(-\varepsilon y,z)=c\gamma(-\varepsilon t,x)\underset{y}{\sum
}\chi(-\varepsilon x+z,y)\gamma(-\varepsilon t,yt^{-1})=$

$c\gamma(-\varepsilon t,x)\frac{\alpha(-t)}{c\gamma(-\varepsilon
t,x-\varepsilon z)}=\frac{\gamma(-\varepsilon t,x)\alpha(-t)}{\gamma
(-\varepsilon t,x)\gamma(-\varepsilon t,-\varepsilon z)\chi(x,zt)}%
=\alpha(-t)\chi(-\varepsilon x,-\varepsilon zt)\gamma(-t^{-1},-\varepsilon
zt),$
\end{enumerate}
from which the result.
\end{proof}

Notice finally that theorem \ref{gwrbasis} may be reworded in more functional analytic  terms as follows.

\begin{theorem}
\label{gwrfunct}
Let $M$   be a finite $A$-right module. Denote $L^2(M)$ the vector space of all complex-valued functions on $M$, endowed with the usual $L^2$ inner product. Set
\begin{enumerate}
	\item $\rho(h_t)(f)(x) = \bar \alpha(t)f(xt)$,  $f\in L^2(M)$ and $t\in A^{\times}$, $x\in M$.
	\item $\rho(u_b))(f)(x) = \gamma(b,x)f(x)$,  $f\in L^2(M)$ and  $b\in A^{sym}$, $x\in M$.
	\item $\rho(w)(f)(x) = c\displaystyle\sum_{y\in M}\chi(-\varepsilon x,y)f(y)$,  $f\in L^2(M)$ and $x\in M$.
\end{enumerate}
(where $\bar \alpha $ denotes the complex conjugate of the character $\alpha$).

These formulas define a unitary linear representation $(L^2(M), \rho)$ of $SL^{\varepsilon}_*(2,A)$, called the generalized Weil representation of  $SL^{\varepsilon}_*(2,A)$ associated to the data 
$(M, \alpha, \gamma, \chi).$
\end{theorem}

\begin{proof} It is enough to verify that the linear operators $\rho_{u_{b}},\rho_{h_{t}},\rho_{w},$ defined as
above, preserve   the universal relations for the generators 
  of $G = SL^{\varepsilon}_*(2,A)$.  This has been done however in the previous theorem, by evaluating on the canonical Dirac delta function basis  $\{  e_x    \}_{x \in  M}. $     of   $L^2(M).$ On the other hand, the unitarity  of the given operators is easily checked.
  
  \end{proof}
    \section{Example 1: The full matrix ring case.}


In this section we take  $A$ to be the full matrix ring $M_n(k),$ where   $k$ denotes the finite field $\mathbb F_q$  with $q$ elements,  endowed with the transpose involution  $\ast$. 

Recall that we have proved in \cite{PSAcom} that  the group $SL^\varepsilon_\ast(2,A)$ has   a Bruhat presentation when $q>3$.

To satisfy the conditions in Theorem \ref{gwrfunct} we may take (keeping the notations introduced there) 
 $M$ to be the $k-$vector space  $Hom_k(E_0, E)$ where  $E_0 = k^n$ and $E$ is a $k-$vector space of   dimension   $m$ endowed with a 
  non-degenerate $k-$quadratic form  $Q_0$  with associated $k-$bilinear form $B_0,$  defined by    $$B_0(u,v) = Q_0(u+v) - Q_0(u) - Q_0(v)$$ for all   $u,v \in E.$  In what follows we will identify 
  $ M = Hom_k(E_0, E)$ with  $ \bigoplus_{1 \leq i \leq  n} E_i, $where  $ E_i = E_0$ for all $i. $
  
  This induces canonically a non degenerate  $A-$valued quadratic form $\mathbb Q$ on $M$ given by
 
\begin{enumerate}
\item  $\mathbb Q(x)_{ii} = Q_0(x_i)   $
\item    $ \mathbb Q(x)_{ij} = B_0(x_i, x_j) $     
\item $  \mathbb Q(x)_{ji} = 0$
\end{enumerate}
for  all $x = (x_1, \dots , x_n) \in  M,  1 \leq i,j \leq n.$

If we pass now to the quotient modulo ``anti-traces", i.e. if we define   
     $\overline { \mathbb Q } = pr \circ \mathbb Q:  M \rightarrow  \bar A, $  
     where  $pr$ denotes the canonical projection of $A$ onto  $A / A^0 $, with  $ A^0 = \{ a - a ^\ast |  a \in  A  \}, $ then we see that 
  $(M, \overline { \mathbb Q } ) $ is a quadratic  module over the involutive ring  $A$ in the sense of Tits \cite{T}.

We fix moreover a non trivial character  $\psi $ of  $k^+$ and we denote by  $tr$ the usual matrix trace from  $A$ onto $k$. Then  $\underline \psi = \psi \circ  tr $ is a non trivial character of  $A^+$ such that   
$   \underline \psi (ab ) = \underline \psi (ba) $  and  $   \underline \psi (a^\ast ) = \underline \psi (a) $  for all  $ a,b \in A.$ On the other hand, we have =
$$   \underline \psi \circ  \mathbb Q  =  \psi \circ  Q  $$
where   $Q$ denotes the  $k-$valued  non-degenerate quadratic form   $  tr \circ \mathbb Q $ over $k,$  whose associated $k-$bilinear form will be denote by $B.$ Notice that    $rank  \;Q = nm$  and in fact
$$  Q(x) = \sum_{1 \leq i \leq n} Q_0(x_i) $$
for  all $x \in  M.$

In what follows we assume  that  $\varepsilon =  -1 $.
We put then for  $s \in A^{sym},  x, y  \in M: $ 

\begin{enumerate}
\item $\gamma(s,x)  =  \underline \psi (s \overline {\mathbb Q} (x) )     $  
\item  $ \chi(x,y) = \underline \psi ( \overline {\mathbb B} (x,y) )   $.   
\end{enumerate}
 
From our  results in  3.3.2  of \cite{sa1} it follows easily that the normalized quadratic Gauss sum  

$$ S_ {\underline\psi, \mathbb Q} (s ) =  \frac{1}{|M|^{\frac{1}{2}}} \sum_{x \in M} \underline \psi (s \mathbb Q (x))   =   S_ {\underline\psi, s \mathbb Q} (1 ) $$    

\noindent
associated to the   matrix valued quadratic form   $\mathbb Q$, defined  for any $s \in    A^{sym} $, 
is constant on the  orbits of the natural action of the multiplicative group   $A^\times$   in  the set  $ A^\times \cap A ^{sym} $, given by  
$ s \mapsto    asa^\ast  \;\;\; (s \in  A^\times \cap A ^{sym},  \; a \in  A^\times).$    
Recall however that there are only two   $A^\times-$orbits in   $A^\times \cap A ^{sym},$  that correspond to the two isomorphy types of non-degenerate  $k-$quadratic forms  of rank $n$, that may be represented by  the unit matrix   $\bf  1 $ in  $A$ and   diagonal matrix $ d_0 =  
diag(1, \dots, 1, t_0$),  for    a fixed non square $t_0 \in A$.

Notice now that  the values of  $ S_{\underline \psi,  \mathbb Q} $   at  $s =  \bf 1  $ and $ s = a_0 $ either coincide or differ by a sign. Indeed, since quadratic Gauss sums associated to an orthogonal sum of   quadratic forms are just the product  of the corresponding single Gauss sums, it may be readily checked that  

$$ \frac{S_ {\underline\psi, \mathbb Q} (d_0 )}{ S_ {\underline\psi, \mathbb Q} (1 )   } =
 \frac{S_ { \psi,  t_0Q_0 } (1 )  }{ S_ { \psi,   Q_0} (1)}  $$ 
 
 \noindent   
so that our statement follows from  the corresponding property for  the classical rank $1$ Gauss sum     $  S_ { \psi,   Q_0}$, for  which is well known or  readily verified. 
Recalling that  $tQ' \simeq Q'$ for any non degenerate  $k-$quadratic form  $Q'$ of even rank, we see then    that 
$$  \frac{S_ { \psi,  Q_0 } (s)   }{ S_ { \psi,   Q_0} (1)} = \alpha(s) $$
 where    $\alpha(s) = 1 $ for $s$ in the orbit of $\bf 1 $ and  $\alpha(s) = (-1)^m $ for  $s$ in the orbit of $d_0,$ i.e.
 \begin{equation}  \label{alpha}
 \alpha(s) = (sgn (det(s)))^m
 \end{equation}
 for all $s \in  A^\times \cap A^{sym}$.  
 Notice that $\alpha^2 = \bf 1$ and $ \alpha= \alpha^{-1} = \bar \alpha $ .
 We put then   
 \begin{equation} \label{c}
  c = \frac{ \alpha(-1)} {   S_ { \psi,   Q   } } 
  \end{equation}
    Recall that we have   $tQ' \simeq Q'$ for any non degenerate  quadratic form  $Q'$ of even rank over the finite field  $k$.

We  check now that $\gamma$ satisfies  \ref{sumagauss}. In fact

\begin{eqnarray*}
c\gamma(t,x)\sum_{y\in M}\chi(x,y)\gamma(t^{-1},y)&=&c\sum_{y\in M}\underline{\psi}(tQ(x)+B(x,y)+t^{-1}Q(y)) \\
\end{eqnarray*}
Setting $x=t^{-1}x'$   we get  $tQ(x)+B(x,y)+t^{-1}Q(y)= t^{-1}Q(x'+y)$. 

Then
\begin{eqnarray*}
c\gamma(t,x)\sum_{y\in M}\chi(x,y)\gamma(t^{-1},y)&=&c\sum_{y\in M}\underline{\psi}(t^{-1}Q(x'+y)\\
&=&cS_{\underline{\psi}\circ Q}(t^{-1})\\
&=&\alpha(-t^{-1})\\
& = &\alpha(-t).
\end{eqnarray*}
as desired.

Summing up,  after having checked the required properties of our data,  taking  the character  $\alpha$ to be given by  \ref{alpha}, for all $s \in A^\times$ and then $c$ to be given by  \ref{c}(see \cite{sa1}, Theorem \ref{gwrfunct}   gives then a  (true) Weil representation  for  $SL^\varepsilon_\ast(2,A) =  Sp(2n,k)$ for any non degenerate $k-$quadratic space $(E_0, Q_0)$ and any choice of the non trivial additive character  $\psi$ of  $k$.  In this way we recover the representation constructed in  \cite{sa1} for even  $m$, besides extending it to the odd $m$ case.

\section{ Example 2: The truncated polynomial ring case}

 We give now an example of application of theorem \ref{gwrfunct} in the case  of a   non semi-simple involutive ring  $(A,*)$ with non-trivial nilpotent Jacobson radical.  
 
Explicitly, we let $k=\mathbb{F}_q$  be the finite field  with $q$ elements, $q$ odd, $m$  a positive integer. We set

\[
 A = A_{m}=k[x]/\left\langle x^{m}\right\rangle = \left\{
{\displaystyle\sum\limits_{i=0}^{m-1}}
a_{i}x^{i}:a_{i}\in k,x^{m}=0\right\}
\]
\noindent 
and we denote by $*$  the $k$-linear involution  on $A_m$ given by $x\mapsto -x$.\\

 We will study here the group   $SL_*^\ast (2,A_m)$  for   $\varepsilon  =  -1 $, which will be denoted  simply $SL_*(2,A_m)$.  
It is known that this group   has a Bruhat presentation \cite{lucho}.
 
 To this end  let us consider the non-degenerate quadratic  $A$-module $(M, Q,B)$,  such that   $M =  A_m, \;\; Q : A_m  \rightarrow A_m $ is given by  $Q(t) = a^*a$ and $B : A_m \times A_m \rightarrow A_m$ is given by $B(a,b) = a^*b+ab^*$.   
 
  Then we have,  for  all    $a,b, t \in A_m: $   
  
  \begin{enumerate}
  
  \item[(i)] 
  $ Q(at)=t^*Q(a)t$; 
   \item[(ii)] 
   
   $Q(a + b)=B(a, b)+ Q(a)+Q(b)$; 
   
    \item[(iii)] $B(at, b)=B(a, bt^*)$; 
    
     \item[(iv)] $B(a,b)=B(b,a)^*$ 
     
      \item[(vi)]  $B(at,b)=t^{\ast}B(a,b)$.
 
 \end{enumerate}

We denote by  $tr$   the linear form on $A_m$  defined by $tr  \left(\sum_{i=0}^{m-1}a_{i}x^{i}\right) = a_{m-1}$.
 
 Then the  form $tr$ is $k$-linear and invariant under the \mbox{involution $\ast$}, i.e., $tr(a^*) = tr(a)$, for any $a\in A_m$. Moreover the $k$-form $tr \circ B$ is a non-degenerate symmetric bilinear form on $A_m$.
 
We fix  be a non-trivial character  $\psi$  of $k^+$ and  we set $\underline{\psi} = \psi \circ tr.$ We assume from now on that  $m$ be is odd. 

The Gauss sum $S_{\underline \psi \circ Q}$ associated to the character  $\underline{\psi} $ and the non-degenerate quadratic  $A_m$-module $(A_m, Q,B)$   is defined by 
\[
S_{\underline \psi  \circ Q}(a) = \sum_{x\in A_m}\underline \psi (aQ(x))\,.
\]
It is known \cite{lucho} that  the function  $\alpha$  from $A_m^{\times}\cap A_m^{sym}$ to $\mathbb{C}^{\times}$ given by  $\  \alpha(a) = \frac{S_{\underline{\psi}\circ Q}(a)}{S_{\underline{\psi}\circ Q}(1)}$ is the sign character of the  group $A_m^{\times}\cap A_m^{sym}$.

Furthermore, we have
 
 \begin{equation}
  \label{alpha1} 
   \alpha(tt^*)=1    \ \ \ \ \ \ \ \ \ \ \ \ \ \ \ \   ( t \in A^\times _m)
 \end{equation} 


   \begin{equation}
   \label{gauss2}
   (S_{\underline{\psi}\circ Q}(1))^2=\alpha(-1)\vert 
A_m\vert
     \end{equation}

\noindent  
  where $\vert A_m\vert$ is the cardinality of the ring $A_m$. 

Let us define the function  $\chi$ from $A_m\times A_m$ to $\mathbb{C}$ as $\chi(a,b)=(\underline{\psi}\circ B)(a,b)$, so     $\chi$ is a symmetric non degenerate  biadditive form  on   $A_m$   such that   $\chi(at,b)=\chi(a,bt^*)$  for all   $a, b, t \in  A_m.$    


\begin{enumerate}
\item $\chi(at,b)=\alpha(tt^*)\chi(a,bt^*)$ for $a, b\in A_m$ and $t\in A_m^{\times}$. \label{c_1}
\item $\chi(b,b)=\chi(-a,b)$ for $a, b\in A_m$, \label{c_2}
\item $\chi(a,b)=1$ for any $a\in A_m$ implies $b=0$.\label{c_3}
\end{enumerate}

It follows  from this and relations  \ref{alpha1}  and  \ref{gauss2} that for the finite  $A_m-$module   $M = A_m$ conditions 1a), 1b) and 1c) of subsection \ref{data} hold.

  We define  now the function  $\gamma$  from $A_m^{sym}\times A_m$ to $\mathbb{C}$ as $\gamma(b,x)=\underline{\psi}(bQ(x))\;\;\;\; (b \in A^sym,  x \in  A_m)$ and we set $c=\frac{\alpha(-1)}{S_{\underline{\psi}\circ Q}(1)}$. 
  Notice that    
  $  c^2\vert A_m\vert=\alpha(-1).   $ Clearly the function  $\gamma$ is   additive in the first variable, satisfies
  
  $$    \gamma(b,xt)=\gamma(tbt^*, x)$$
   for $t\in A_m^{\times}, x \in A_m, b\in A_m^{sym}$ \label{cc_2}
   and relates to    $\chi$ through
  
  \begin{equation}  
  \label{cc_3}
  \gamma(b,x+y)=\gamma(b,x)\gamma(b,y)\chi(x, yb) 
\end{equation}

\noindent
for any $x, y\in A_m, b\in A_m^{sym}$. 

Finally 
 
\begin{equation}
\label{sumagauss}
c\gamma(t,a)\sum_{d\in A_m}\chi(a,d)\gamma(t^{-1},d)=\alpha(-t),
\end{equation}
for any symmetric element $t\in A_m^{\times}$, where $c\in\mathbb{C}^{\times}$ satisfies 

\begin{equation} \label{cc_4}
c^2\vert A_m\vert=\alpha(-1).   
\end{equation}

The verification that  $\gamma$ satisfies  \ref{sumagauss} is completely analogous to the one in the preceding example.

Summing up, the above setup provides  a Weil representation of  
$SL_\ast^\varepsilon (2,A_m) $ in $L^2(M)$,     according to theorem \ref{gwrfunct} (which is exactly the one constructed in \cite{lucho}).

\section{A first decomposition of the Weil representation} 
 We give here a first decomposition of the Weil representation$ (L^2(M), \rho)$ of $G$, taking advantage of the fact that there is  a group of intertwinning operators that acts naturally in $L^2(M),$ to wit, the ``unitary group"  $U(\gamma,\chi)$ of the pair $(\gamma, \chi).$

\begin{definition}  
We denote by $U(\gamma,\chi)$ the group of all $A$-linear automorphisms $\varphi$ of $M$   such that

\begin{enumerate}
\item  $\gamma(b,\varphi(x))=\gamma(b,x)$ for any $b\in A^{sym}$, $x\in M$;

\item   $\chi(\varphi(x),\varphi(y))=\chi(x, y)$ for any $x, y\in M$.
\end{enumerate} 
\end{definition}

\begin{remark}
 
Condition  (2) for   $\varphi$ in $U(\gamma,\chi)$ is only necessary in the rather peculiar case where there are no $\varepsilon-$symmetric invertible elements    in $A,$ because  if  $t \in A^\times \cap A^{sym}, $ then we   have, for all  $x,y \in M,$  
\begin{equation}  \label{gamadachi}
\chi(x,y)=\frac{\gamma(t,x+yt^{-1})}{\gamma(t,x)\gamma(t,yt^{-1})}=\frac{\gamma(t,\varphi(x)+\varphi(y)t^{-1})}{(\gamma(t,\varphi(x))\gamma(t,\varphi(y)t^{-1})}=\chi(\varphi(x),\varphi(y)).
\end{equation}
The next lemma addresses the converse question. 

\end{remark}
 
\begin{lemma}
 
We can recover the function $\gamma $ from   $\chi$ as follows.

\begin{equation}
\label{chidagama}
  \gamma( t, x) =    \chi (x2^{-1}t^\ast,  x)          \;\;\;\;\;\;\;\;\;\;\;\;\  ( t \in A^\times \cap A^{sym}, x \in M)          
\end{equation}
 \end{lemma}
\begin{proof}
 To prove relation  \ref{chidagama} we calculate  $\gamma(r, xs + xs), $ for  $x\in M,  r,s \in A^\times \cap A^{sym}$,  in two different ways:
 
  $\gamma(r, xs + xs)  = \gamma(2srs^\ast, x)\chi(xs,xsr), $

  $\gamma(r, xs + xs)  = \gamma (2sr2s^\ast, x) = \gamma (4srs^\ast, x).$
  
  Therefore
  $$ \gamma(2srs^\ast,x) = \chi(xs,xsr) $$
  and choosing  $r,s \in   A^\times \cap A^{sym}  $ such that  $sr = 1$ we get 
  $$  \gamma(2s^\ast, x) = \chi(xs, x) $$
  Putting  $2s^\ast = t,$ we finally get 
  $ \gamma(t,x) = \chi(x2^{-1}t^\ast,  x).   $

 \end{proof}
\begin{proposition}
The group $ \Gamma = U(\gamma,\chi)$ acts naturally on $       L^2(M)       $ by $(\varphi.f)(x)=f(\varphi^{-1}(m))$. This action commutes with the Weil representation $\rho$ of $G = SL^{\varepsilon}_*(2,A)$
\end{proposition}
 
\begin{proof}     Using the explicit definition in terms of $\chi$ and $\gamma,$ of the Weil operators $\rho(g)$ for our generators  $g$ of $G$, besides the fact that  all $\varphi \in \Gamma$ are   $A-$linear,  one readily checks that the natural action of  $\Gamma$ commutes with them.  
\end{proof}

\begin{definition}
Let $(\pi,V)$ be an irreducible representation of $\Gamma = U(\gamma,\chi)$. Denote by $(L^2_V(M)[\pi], \rho)$ the representation of $G$ in the space  $L^2_V(M)[\pi]$ consisting  of all $V-$valued functions  $f $ on $M$ such that $f(\varphi(x))=\pi(\varphi))(f(x)),$ whose action is the action  of the Weil representation $\rho$ of $SL^{\varepsilon}_*(2,A)$.
\end{definition}
Now the well known description of isotypical components for the action of one group that intertwines with the action of another group in the same space  (see \cite{sa1}, for instance) becomes:
\begin{proposition}
The Weil     $ (L^2(M), \rho)$ of $G$       decomposes as the direct sum of the isotypical components of the natural representation of   $\Gamma$    in   $L^2(M)$ as follows:

      $$    W  =  \bigoplus _{\pi \in  \widehat \Gamma }  \pi^\vee \otimes L^2_V(M)[\pi], $$
 \noindent
where $ \pi^\vee $ denotes the contragredient of the ireducible  representation $\pi \in  $ of  $  \widehat \Gamma.$

\end{proposition}

 The irreducibility of the representations $(L^2_V(M)[\pi], \rho)$ of $G$ constructed in this way,  remains an open problem in general (see \cite{sa1} for a complete answer for  $G = Sp(4, \mathbb F_q)$). We expect to address this question in other cases  elsewhere.


\bigskip
\noindent
{\bf Acknowledgments:} We thank P. Cartier for inspiring discussions on this subject, stretching over several decades...

\end{document}